\documentclass[MSNbibl,number,citesort,seceqn,dvips]{arxbj}
\usepackage{multirow}

\aid{0}
\volume{19}
\issue{5A}
\pubyear{2013}
\firstpage{1790}
\lastpage{1817}
\doi{10.3150/12-BEJ430} 

\makeatletter

\define@key{bid}{pmid}{}

\newproclaim{definition}[proposition]{Definition}
\newremark{remark}{Remark}[section]
\newtheorem{theorem}[proposition]{Theorem}
\newtheorem{proposition}{Proposition}[section]
\newtheorem{corollary}[proposition]{Corollary}
\newtheorem{lemma}[proposition]{Lemma}

\newcommand{\eqref}[1]{(\ref{#1})}
\newcommand{\pp}{\mbox{$\mathbf p$}}
\makeatother

\begin{document}
\begin{frontmatter}

\title{Theory of self-learning $Q$-matrix}
\runtitle{Theory of self-learning $Q$-matrix}

\begin{aug}
\author{\fnms{Jingchen} \snm{Liu}\thanksref{e1}\ead[label=e1,mark]{jcliu@stat.columbia.edu}},
\author{\fnms{Gongjun} \snm{Xu}\corref{}\thanksref{e2}\ead[label=e2,mark]{gongjun@stat.columbia.edu}} \and
\author{\fnms{Zhiliang} \snm{Ying}\thanksref{e3}\ead[label=e3,mark]{zying@stat.columbia.edu}}
\runauthor{J. Liu, G. Xu and Z. Ying} 
\address{Department of Statistics, Columbia University, 1255 Amsterdam
Avenue, New York, NY
10027, USA. \printead{e1,e2,e3}}
\end{aug}

\received{\smonth{3} \syear{2011}}
\revised{\smonth{11} \syear{2011}}

%
\begin{abstract}
Cognitive assessment is a growing area in psychological and
educational measurement, where tests are given to assess
mastery/deficiency of attributes or skills. A key issue is the
correct identification of attributes associated with items in a
test. In this paper, we set up a mathematical framework under which
theoretical properties may be discussed. We establish sufficient
conditions to ensure that the attributes required by each item are
learnable from the data.
\end{abstract}


\end{frontmatter}

\section{Introduction}\label{SecIntro}

Cognitive diagnosis has recently gained prominence in educational
assessment, psychiatric evaluation, and many other disciplines. A
key task is the correct specification of item-attribute
relationships. A widely used mathematical formulation is the well
known $Q$-matrix \cite{Tatsuoka1983}. Under the setting of the
$Q$-matrix, a typical modeling approach
assumes a latent variable structure in which each subject possesses
a vector of $k$ attributes and responds to $m$ items. The so-called
$Q$-matrix is an $m\times k$ binary matrix establishing the
relationship between responses and attributes by indicating the
required attributes for each item. The entry in the $i$th row and
$j$th column indicates if item $i$ requires attribute $j$ (see
Example~\ref{SecEx} for a demonstration of a $Q$-matrix). A short list of
further developments of cognitive diagnosis models (CDMs) based on
the $Q$-matrix includes the rule space method \cite{Tatsuoka1985,Tatsuoka},
the reparameterized unified/fusion model (RUM) \cite
{DiBello,Hartz,He}, the conjunctive (noncompensatory) DINA and NIDA models
\cite{Junker,TatsuokaC,dela,Templin2006,dela2008}, the
compensatory DINO and NIDO models \cite{Templin,Templin2006},
the attribute
hierarchy method \cite{AHM}, and clustering methods \cite{Chiu}; see
also \cite{Junker0,von,Rupp} for more approaches to cognitive
diagnosis.


Statistical analysis
with CDMs typically
assumes a known $Q$-matrix provided by experts such as those who
developed the questions \cite{Rupp2002,Henson,RoussosTH,Stout2007}.
Such a priori knowledge when correct is certainly
very helpful for both model estimation and eventually identification
of subjects' latent attributes. On the other hand, model fitting is
usually sensitive to the choice of $Q$-matrix and its
misspecification could seriously affect the goodness of fit. This is
one of the main sources for lack of fit. Various diagnostic tools
and testing procedures have been developed \cite{Rupp20082,dela2,Henson0,Liu,He2}. A~comprehensive review of
diagnostic classification
models can be found in \cite{Rupp2008}.

Despite the importance of the $Q$-matrix in cognitive diagnosis, its
estimation problem is largely an unexplored area. Unlike typical
inference problems, the inference for the $Q$-matrix is particularly
challenging for the following reasons. First, in many cases, the
$Q$-matrix is simply nonidentifiable. One typical situation is that
multiple $Q$-matrices lead to an identical response distribution.
Therefore, we only expect to identify the $Q$-matrix up to some
equivalence relation (Definition~\ref{DefEq}). In other words, two
$Q$-matrices in the same equivalence class are not distinguishable
based on data. Our first task is to define a meaningful and
identifiable equivalence class. Second, the $Q$-matrix lives on a
discrete space -- the set of $m\times k$ matrices with binary
entries. This discrete nature makes analysis particularly difficult
because calculus tools are not applicable. In fact, most analyses
are combinatorics based. Third, the model makes explicit
distributional assumptions on the (unobserved) attributes, dictating
the law of observed responses. The dependence of responses on
attributes via $Q$-matrix is a highly nonlinear discrete function.
The nonlinearity also adds to the difficulty of the analysis.

The primary purpose of this paper is to provide theoretical analyses
on the learnability of the underlying $Q$-matrix. In particular, we
obtain definitive answers to the identifiability of $Q$-matrix for
one of the most commonly used models -- the DINA model -- by
specifying a set of sufficient conditions under which the $Q$-matrix
is identifiable up to an explicitly defined equivalence class. We
also present the corresponding consistent estimators. We believe
that the results (especially the intermediate results) and analysis
strategies can be extended to other conjunctive models
\cite{Maris,Junker,Templin2006,Templin,HRHT}.

The rest of this paper is organized as follows. In Section~\ref{SecDINA}, we present the basic inference result for
$Q$-matrices in a conjunctive model with no slipping or guessing. In
addition, we introduce all the necessary terminologies and technical
conditions. In Section~\ref{SecProb}, we extend the results in
Section~\ref{SecDINA} to the DINA model with known slipping and
guessing parameters. In Section~\ref{SecSlip}, we further generalize
the results to the DINA model with unknown slipping parameters.
Further discussion is provided in Section~\ref{SecDis}. Proofs are
given in Section~\ref{SecProof}. Lastly, the proofs of two key
propositions are given in the \hyperref[app]{Appendix}.

\section{Model specifications and basic results}\label{SecDINA}

We start the discussion with a simplified situation, under which the
responses depend on the attribute profile deterministically (with no
uncertainty). We describe our estimation procedure under this simple
scenario. The results for the general cases are given in Sections~\ref{SecProb} and~\ref{SecSlip}.

\subsection{Basic model specifications}\label{SecModel}

The model specifications consist of the following concepts.

\textit{Attributes}: subject's (unobserved) mastery of certain
skills. Consider that there are $k$ attributes. Let $\mathbf A =
(A^1,\ldots,A^k)^\top$ be the vector of attributes and $A^j\in\{0,1\}$
be the indicator of the presence or absence of the $j$th
attribute.\eject

\textit{Responses}: subject's binary responses to items. Consider
that there are $m$ items. Let $\mathbf R=(R^1,\ldots,R^m)^\top$ be the
vector of responses and $R^i\in\{0,1\}$ be the response to the
$i$th item.

Both $\mathbf A$ and $\mathbf R$ are subject specific. We assume
that the integers $m$ and $k$ are known.

\textit{$Q$-matrix}: the link between item responses and attributes.
We define an $m\times k$ matrix $Q=(Q_{ij})_{m\times k}$. For each
$i$ and $j$, $Q_{ij}=1$ when item $i$ requires attribute $j$ and $0$
otherwise.

Furthermore, we define
%
\begin{equation}\label{Ability}
\xi^i = \prod_{j=1}^k (A^j)^{Q_{ij}} = {\mathbf1} (A^j\geq Q_{ij}\dvtx j
=1,\ldots,k),
\end{equation}
which indicates whether a subject with attribute $\mathbf A$ is
capable of providing a positive response to item $i$. This model is
conjunctive, meaning that it is necessary and sufficient to master
all the necessary skills to be capable of solving one problem.
Possessing additional attributes does not compensate for the absence
of necessary attributes. In this section, we consider the simplest
situation that there is no uncertainty in the response, that is,
%
\begin{equation}\label{Perf}
R^i = \xi^i
\end{equation}
for $i=1,\ldots,m$. Therefore, the responses are completely determined
by the attributes. We assume that all items require at least one
attribute. Equivalently, the $Q$-matrix does not have zero row
vectors. Subjects who do not possess any attribute are not capable
of responding positively to any item.

We use subscripts to indicate different subjects. For instance,
$\mathbf R_r = (R^1_r,\ldots,R^m_r)^\top$ is the response vector of
subject $r$. Similarly, $\mathbf A_r$ is the attribute vector of
subject $r$. We observe $\mathbf R_1,\ldots,\mathbf R_N$, where we use
$N$ to denote sample size. The attributes $\mathbf A_r$ are not
observed. Our objective is to make inference on the $Q$-matrix based
on the observed responses.

\subsection{Estimation of $Q$-matrix}

We first introduce a few quantities for the presentation of an
estimator.

\subsubsection*{$T$-matrix}
In order to provide an estimator of $Q$, we
first introduce one central quantity, the $T$-matrix, which connects
the $Q$-matrix with the response and attribute distributions. Matrix
$T(Q)$ has $2^{k}-1$ columns each of which corresponds to one
nonzero attribute vector, $\mathbf A\in\{0,1\}^{k}\setminus
\{(0,\ldots,0)\}$. Instead of labeling the columns of $T(Q)$ by ordinal
numbers, we label them by binary vectors of length $k$. For
instance, the $\mathbf A$th column of $T(Q)$ is the column that
corresponds to attribute $\mathbf A$, for all $\mathbf A\neq
(0,\ldots,0)$.

Let $I_i$ be a generic notation for positive responses to item $i$.
Let ``$\wedge$'' stand for ``and'' combination. For instance,
$I_{i_1}\wedge I_{i_2}$ denotes positive responses to both items
$i_1$ and $i_2$. Each row of $T(Q)$ corresponds to one item or one
``and'' combination of items, for instance, $I_{i_1}$,
$I_{i_1}\wedge I_{i_2}$ or $I_{i_1}\wedge I_{i_2}\wedge I_{i_3}$,
\ldots. If $T(Q)$ contains all the single items and all ``and''
combinations, $T(Q)$ contains $2^m-1$ rows. We will later say that
such a $T(Q)$ is \emph{saturated} (Definition~\ref{DefSat} in
Section~\ref{SecRes1}).

We now describe each row vector of $T(Q)$. We define that
$B_{Q}(I_{i})$ is a $2^{k}-1$ dimensional row vector. Using the same
labeling system as that of the columns of $T(Q)$, the $\mathbf A$th
element of $B_{Q}(I_{i})$ is defined as
$\prod_{j=1}^{k}(A^{j})^{Q_{ij}}$, which indicates if a subject with
attribute $\mathbf A$ is able to solve item $i$.

Using a similar notation, we define that
%
\begin{equation}\label{ProdB}
B_{Q}(I_{i_{1}}\wedge\cdots\wedge I_{i_{l}})=\Upsilon
_{h=1}^{l}B_{Q}(I_{i_{h}}),
\end{equation}
where the operator ``$\Upsilon_{h=1}^{l}$'' is element-by-element
multiplication from $B_Q(I_{i_1})$ to $B_Q(I_{i_l})$. For instance,
\[
W=\Upsilon_{h=1}^l V_h
\]
means that $W^j = \prod_{h=1}^l V_{h}^j$, where
$W=(W^1,\ldots,W^{2^k-1})$ and $V_h=(V_h^1,\ldots,V_h^{2^k-1})$.
Therefore, $B_Q(I_{i_1}\wedge\cdots\wedge I_{i_l})$ is the vector
indicating the attributes that are capable of responding positively
to items $i_1,\ldots,i_l$. The row in $T(Q)$ corresponding to
$I_{i_{1}}\wedge\cdots\wedge I_{i_{l}}$ is $B_{Q}(I_{i_{1}}\wedge
\cdots\wedge I_{i_{l}})$.

\subsubsection*{$\alpha$-vector}
We let $\alpha$ be a column vector the
length of which equals to the number of rows of $T(Q)$. Each element
of $\alpha$ corresponds to one row vector of $T(Q)$. The element in
$\alpha$ corresponding to $I_{i_{1}}\wedge\cdots\wedge I_{i_{l}}$ is
defined as $N_{I_{i_{1}}\wedge\cdots\wedge I_{i_{l}}}/N$, where
$N_{I_{i_{1}}\wedge\cdots\wedge I_{i_{l}}}$ denotes the number of
people who have positive responses to items $i_1,\ldots,i_l$, that is
\[
N_{I_{i_{1}}\wedge\cdots\wedge I_{i_{l}}}= \sum_{r=1}^N
I(R_r^{i_j}=1\dvtx j=1,\ldots,l).
\]

For each $\mathbf A\in\{0,1\}^k$, we let
%
\begin{equation}\label{phat}\hat p_{\mathbf A} = \frac1 N\sum
_{r=1}^N I(\mathbf A_r = \mathbf A).
\end{equation}
If \eqref{Perf} is strictly respected, then
%
\begin{equation}\label{id}
T(Q)\hat{\mathbf p} = \alpha,
\end{equation}
where $\hat{\mathbf p}= (\hat p_{\mathbf A}\dvtx  \mathbf A\in\{0,1\}^k
\setminus \{(0,\ldots,0)\})$ is arranged in the same order as the
columns of $T(Q)$. This is because each row of $T(Q)$ indicates the
attribute profiles corresponding to subjects capable of responding
positively to that set of item(s). Vector $\hat{\mathbf p}$ contains
the proportions of subjects with each attribute profile. For each
set of items, matrix multiplication sums up the proportions
corresponding to each attribute profile capable of responding
positively to that set of items, giving us the total proportion of
subjects who respond positively to the corresponding items.

\subsubsection*{The estimator of the $Q$-matrix}
 For each $m\times k$
binary matrix $Q'$, we define
%
\begin{equation}\label{basic}
S(Q')=\inf_{\mathbf p\in[0,1]^{2^k-1}}|T(Q')\mathbf p-\alpha|,
\end{equation}
where $\mathbf p =(p_{\mathbf A}\dvtx  \mathbf A \neq(0,\ldots,0))$. The
above minimization is subject to the constraint that $\sum_{\mathbf
A \neq\mathbf(0,\ldots,0)}p_{\mathbf A} \in[0,1]$. $|\cdot|$ is the
Euclidean distance. An estimator of $Q$ can be obtained by
minimizing $S(Q')$,
%
\begin{equation}\label{EstVan}
\hat Q = \arg\inf_{Q'} S(Q'),
\end{equation}
where ``$\arg\inf$'' is the minimizer of the minimization problem
over all $m\times k$ binary matrices. Note that the minimizers are
not unique. We will later prove that the minimizers are in the same
meaningful equivalence class. Because of \eqref{id}, the true
$Q$-matrix is always among the minimizers because $S(Q)=0$.

\subsection{Example}\label{SecEx}

We illustrate the above construction by one simple example. We
emphasize that this example is discussed to explain the estimation
procedure for a concrete and simple example. The proposed estimator
is certainly able to handle much larger $Q$-matrices. We consider the
following $3\times2$ $Q$-matrix,
%
\begin{equation}\label{Q}
Q=\,
\begin{tabular}{@{}lll@{}}
\hline
& Addition & Multiplication\\
\hline
$2+3$ & $1$ & $0$ \\
$5\times2$ & $0$ & $1$ \\
$(2+3)\times2$ & $1$ & $1$\\
\hline
\end{tabular}
\end{equation}
There are two attributes and three items. We consider the
contingency table of attributes,\vspace*{6pt}
\begin{center}
\begin{tabular}{@{}lll@{}}
\hline
& \multicolumn{2}{l@{}}{Multiplication} \\
\hline
\multirow{2}{*}{Addition} & $\hat p_{00}$ & $\hat p_{01}$
\\
& $\hat p_{10}$ & $\hat p_{11}$
\\
\hline
\end{tabular} \vspace*{6pt}
\end{center}
In the above table, $\hat p_{00}$ is the proportional of people who
do not master either addition or multiplication. Similarly, we
define $\hat p_{01}$, $\hat p_{10}$ and $\hat p_{11}$. $\{\hat
p_{ij}; j=0,1\}$ is not observed.

Just for illustration, we construct a simple nonsaturated $T$-matrix.
Suppose the relationship in (\ref{Perf}) is strictly respected. Then,
we should be able to establish the following identities:
%
\begin{equation}\label{margin}
N(\hat p_{10}+\hat p_{11})=N_{I_{1}},\qquad N(\hat p_{01}+\hat
p_{11})=N_{I_{2}},\qquad N\hat p_{11}=N_{I_{3}}.
\end{equation}
Therefore, if we let $\hat{\mathbf p}=(\hat p_{10},\hat p_{01},\hat
p_{11})$, the above display imposes three linear constraints on the
vector $\hat{\mathbf p}$. Together with the natural constraint that
$\sum_{ij}\hat p_{ij}=1$, $\hat{\mathbf p}$ solves linear equation,
%
\begin{equation}\label{Linear}
T(Q)\hat{\mathbf p}=\alpha,
\end{equation}
subject to the constraints that $\hat{\mathbf p}\in\lbrack
0,1]^{3}$ and $\hat p_{10}+ \hat p_{01}+\hat p_{11}\in[0,1]$, where
%
\begin{equation}\label{Num}
T(Q)= \pmatrix{
1 & 0 & 1 \cr
0 & 1 & 1 \cr
0 & 0 & 1
} ,\qquad \alpha= \pmatrix{
N_{I_{1}}/N \cr
N_{I_{2}}/N \cr
N_{I_{3}}/N
} .
\end{equation}
Each column of $T(Q)$ corresponds to one attribute profile. The
first column corresponds to $\mathbf A=(1,0)$, the second column to
$\mathbf A=(0,1)$, and the third column to $\mathbf A=(1,1)$. The
first row corresponds to item $2+3$, the second row to $5\times2$
and the last row to $(2+3)\times2$. For this particular situation,
$T(Q)$ has full rank and there exists one unique solution to
\eqref{Linear}. In fact, we would not expect the constrained
solution to the linear equation in \eqref{Linear} to always exist
unless \eqref{Perf} is strictly followed. This is the topic of the
next section.

The identities in (\ref{margin}) only consider the marginal rate of
each question. There are additional constraints if one considers
``combinations'' among items. For instance,
\[
N\hat p_{11}=N_{I_{1}\wedge I_{2}}.
\]
People who are able to solve problem 3 must have both attributes and
therefore are able to solve both problems 1 and 2. Again, if
\eqref{Perf} is not strictly followed, this is not necessarily
respected in the real data, though it is a logical conclusion. The DINA
in the next section handles such a case. Upon
considering $I_1$, $I_2$, $I_3$ and $I_1 \wedge I_2$, the new
$T$-matrix is
%
\begin{equation}\label{NewT}
T(Q)= \pmatrix{
1 & 0 & 1 \cr
0 & 1 & 1 \cr
0 & 0 & 1 \cr
0 & 0 & 1%
} ,\qquad \alpha= \pmatrix{
N_{I_{1}}/N \cr
N_{I_{2}}/N \cr
N_{I_{3}}/N \cr
N_{I_{1}\wedge I_{2}}/N%
} .
\end{equation}
The last row is added corresponding to $I_1 \wedge I_2$. With
\eqref{Perf} in force, we have
%
\begin{equation}
S(Q)=\inf_{\mathbf p\in[0,1]^3}|T(Q)\mathbf p-\alpha| = |T(Q)\hat
{\mathbf p}-\alpha|=0,
\end{equation}
if $Q$ is the true matrix.

\subsection{Basic results}\label{SecRes1}

Before stating the main result, we provide a list of notations,
which will be used in the discussions.
\begin{itemize}
\item Linear space spanned by vectors $V_1,\ldots,V_l$:
\[
\mathcal L (V_1,\ldots,V_l)= \Biggl\{\sum_{j=1}^l a_j V_j\dvtx  a_j \in\mathbb
R \Biggr\}.
\]

%
\item For a matrix $M$, $M_{1:l}$ denotes the submatrix containing the
first $l$ rows and all columns of~$M$.
\item Vector $e_i$ denotes a column vector such that the $i$th element
is one and the rest are zero. When there is no ambiguity, we omit the
length index of $e_i$.
\item Matrix $\mathcal I_l$ denotes the $l\times l$ identity matrix.
\item For a matrix $M$, $C(M)$ is the linear space generated by the
column vectors of $M$. It is usually called the \emph{column space} of $M$.
\item$C_M$ denotes the set of column vectors of $M$.
\item$R_M$ denotes the set of row vectors of $M$.
\item Vector $\mathbf0$ denotes the zero vector, $(0,\ldots,0)$. When
there is no ambiguity, we omit the index of length.
\item Scalar $p_{\mathbf A}$ denotes the probability that a subject has
attribute profile $\mathbf A$. For instance, $p_{10}$ is the
probability that a subject has attribute one but not attribute two.
\item Define a $2^k-1$ dimensional vector
\[
\mathbf p= (p_{\mathbf A}\dvtx  \mathbf A \in\{0,1\}^k\setminus \{\mathbf
0\} ).
\]

%
\item Let $c$ and $g$ be two $m$ dimensional vectors. We write $c\succ
g$ if $c_i> g_i$ for all $1\leq i\leq m$.
\item We write $c\ncong g$ if $c_i \neq g_i$ for all $i=1,\ldots,m$.
\item Matrix $Q$ denotes the true matrix and $Q'$ denotes a generic
$m\times k$ binary matrix.
\end{itemize}

The following definitions will be used in subsequent discussions.

\begin{definition}\label{DefSat}
We say that $T(Q)$ is \emph{saturated} if all combinations of form
$I_{i_1}\wedge\cdots\wedge I_{i_l}$, for $l=1,\ldots,m$, are
included in
$T(Q)$.
\end{definition}

\begin{definition}\label{DefEq}
We write $Q\sim Q'$ if and only if $Q$ and $Q'$ have identical
column vectors, which could be arranged in different orders;
otherwise, we write $Q\nsim Q'$.
\end{definition}

\begin{remark}
It is not hard to show that ``$\sim$'' is an \emph{equivalence
relation}. $Q\sim Q'$ if and only if they are identical after an
appropriate permutation of the columns. Each column of $Q$ is
interpreted as an attribute. Permuting the columns of $Q$ is
equivalent to relabeling the attributes. For $Q\sim Q'$, we are not
able to distinguish $Q$ from $Q'$ based on data.
\end{remark}

\begin{definition}\label{DefComp}
A $Q$-matrix is said to be \emph{complete} if
$\{e_i\dvtx i=1,\ldots,k\}\subset R_Q$ ($R_Q$ is the set of row vectors of
$Q$); otherwise, we say that $Q$ is \emph{incomplete}.
\end{definition}

A $Q$-matrix is complete if and only if for each attribute there
exists an item only requiring that attribute. Completeness implies
that $m\geq k$. We will show that completeness is among the
sufficient conditions to identify $Q$.

\begin{remark}
One of the main objectives of cognitive assessment is to identify
the subjects' attributes; see \cite{Rupp2008} for other
applications. It has been established in \cite{Chiu} that the
completeness of the $Q$-matrix is a sufficient and necessary condition
for a set of items to consistently identify attributes if \eqref{Perf}
is strictly followed.
Thus, it is usually recommended to use a complete $Q$-matrix. For a
precise formulation, see \cite{Chiu}.
\end{remark}



Listed below are assumptions which will be used in subsequent
development.
\begin{enumerate}[C3]
\item[C1] $Q$ is complete.
\item[C2] $T(Q)$ is saturated.
\item[C3] $\mathbf A_1,\ldots,\mathbf A_N$ are i.i.d. random vectors
following distribution
\[
P(\mathbf A_r=\mathbf A) = p^*_{\mathbf A}.
\]
\end{enumerate}

We further let $\mathbf p^* = (p^*_{\mathbf A}\dvtx  \mathbf A \in\{0,1\}
\setminus \{\mathbf0\})$.
\begin{enumerate}[C5]
\item[C4] $(p_{\mathbf0}^*, \mathbf p^*) \succ\mathbf0$.
\item[C5] Each attribute has been required by at least two items.
\end{enumerate}

With these preparations, we are ready to introduce the first
theorem, the proof of which is given in Section~\ref{SecProof}.

\begin{theorem}\label{ThmVan}
Assume that conditions \textup{C1}--\textup{C5} are in force. Suppose that for subject
$r$ the response corresponding to item $i$ follows
\[
R^i_r=\xi^i_r = \prod_{j=1}^k (A^j_r)^{Q_{ij}}.
\]
Let $\hat Q$, defined in \eqref{EstVan}, be a minimizer of $S(Q')$
among all $m \times k$ binary matrices, where $S(Q')$ is defined in
\eqref{basic}. Then,
%
\begin{equation}\label{EstQ}\lim_{N\rightarrow\infty}P(\hat Q\sim Q)=1.
\end{equation}
Further, let
%
\begin{equation}\label{Estp}\tilde{\mathbf p} = \arg\inf_{\mathbf
p} |T(\hat Q)\mathbf p -\alpha|^2.
\end{equation}
With an appropriate rearrangement of the columns of $\hat Q$, for
any $\varepsilon>0$
\[
\lim_{N\rightarrow\infty}P( |\tilde{\mathbf p} - \mathbf p^*|\leq
\varepsilon)=1.
\]
\end{theorem}

\begin{remark}If $Q_1 \sim Q_2$, the two matrices only differ by a
column permutation and will be considered to be the ``same''.
Therefore, we expect to identify the equivalence class that $Q$ belongs
to. Also, note that $S(Q_1)=S(Q_2)$ if $Q_1 \sim Q_2$.
\end{remark}

\begin{remark}
In order to obtain the consistency of $\hat Q$ (subject to a column
permutation), it is necessary to have $\mathbf p^*$ not living on
some sub-manifold. To see a counter example, suppose that $P(\mathbf
A_r = (1,\ldots,1)^\top)=p^*_{1\ldots1}=1$. Then, for all $Q$,
$P(\mathbf
R_r=(1,\ldots,1)^\top)=1$, that is, all subjects are able to solve all
problems. Therefore, the distribution of $\mathbf R$ is independent
of $Q$. In other words, the $Q$-matrix is not identifiable. More
generally, if there exit $A_r^i$ and $A_r^j$ such that
$P(A_r^i=A_r^j)=1$, then the $Q$-matrix is not identifiable based on
the data. This is because one cannot tell if an item requires
attribute $i$ alone, attribute $j$ alone, or both; see
\cite{Reckase,ReckaseMIRT} for similar cases for the multidimensional IRT
models.
\end{remark}

\begin{remark}
Note that the estimator of the attribute distribution, $\tilde
{\mathbf p}$, in \eqref{Estp} depends on the order of columns of
$\hat Q$. In order to achieve consistency, we will need to arrange
the columns of $\hat Q$ such that $\hat Q =Q$ whenever $\hat Q \sim
Q$.
\end{remark}

\begin{remark}\label{RemSplit}
One practical issue associated with the proposed procedure is the
computation. For a specific $Q$, the computation of $S(Q)$ only
involves a constraint minimization of a quadratic function. However,
if $m$ or $k$ is large, the computation overhead of searching the
minimizer of $S(Q)$ over the space of $m\times k$ matrices could be
substantial. One practical solution is to break the $Q$-matrix into
smaller sub-matrices. For instance, one may divide the $m$ items in
to $l$ groups (possibly with nonempty overlap across different
groups). Then apply the proposed estimator to each of the $l$ group
of items. This is equivalent to breaking a big $m$ by $k$ $Q$-matrix
into several smaller matrices and estimating each of them
separately. Lastly, combine the $l$ estimated sub-matrices together
to form a single estimate. The consistency results can be applied to
each of the $l$ sub-matrices and therefore the combined matrix is
also a consistent estimator. A similar technique has been discussed
in Chapter 8.6 of \cite{Tatsuoka}.
\end{remark}

\begin{remark}
Conditions C1 and C2 are imposed to guarantee consistency. They may not
be always necessary. Furthermore, constructing a saturated $T$-matrix
is sometimes computationally not feasible, especially when the number
of items is large. In practice, one may include the combinations of one
item, two items,\ \ldots, $j$ items. The choice of $j$ depends on the
sample size and the computation resources.
The condition C5 is required for technical purposes. Nonetheless, one
can in fact construct counterexamples, that is, the $Q$-matrix is not
identifiable up to the relationship ``$\sim$'', if C5 is violated.
\end{remark}


\section{DINA model with known slipping and guessing parameters}\label
{SecProb}

\subsection{Model specification}\label{SecF}

In this section, we extend the inference results in the previous
section to the situation under which the responses do not
deterministically depend on the attributes. In particular, we
consider the DINA (Deterministic Input, Noisy Output ``AND'' gate)
model \cite{Junker}.
We would like to introduce two parameters: the slipping parameter
($s_i$) and the guessing parameter $(g_i)$. Here $1-s_i$ ($g_i$) is
the probability of a subject's responding positively to item $i$
given that s/he is capable of solving that problem. To simplify the
notations, we denote $1-s_i$ by $c_i$. An extension of \eqref{Perf}
to include slipping and guessing specifies the response
probabilities as
%
\begin{equation}\label{Prob}P(R^i = 1|\xi^i) = c_i^{\xi^i}
g_i^{1-\xi^i},
\end{equation}
where $\xi^i$ is the capability indicator defined in
\eqref{Ability}. In addition, conditional on $\{\xi^1,\ldots,\xi^m\}$,
$\{R^1,\ldots,R^m\}$ are jointly independent.


In this context, the $T$-matrix needs to be modified accordingly.
Throughout this section, we assume that both $c_i$'s and $g_i$'s are
known. We discuss the case that $c_i$'s are unknown in the next
section.

We first consider the case that $g_i = 0$ for all $i=1,\ldots,m$.
We introduce a diagonal matrix $D_{c}$. If the $h$th row of matrix
$T_{c}(Q)$ corresponds to $I_{i_{1}}\wedge\cdots\wedge I_{i_{l}}$,
then the $h$th diagonal element of $D_{c}$ is $c_{i_{1}}\times
\cdots\times c_{i_{l}}$. Then, we let
%
\begin{equation}\label{As}
T_{c}(Q)=D_{c}T(Q),
\end{equation}
where $T(Q)$ is the binary matrix defined previously. In other
words, we multiply each row of $T(Q)$ by a common factor and obtain
$T_{c}(Q)$. Note that in absence of slipping ($c_i=1$ for each $i$)
we have that $T_c(Q)= T(Q)$.

There is another equivalent way of constructing $T_{c}(Q)$. We define%
\[
B_{c,Q}(I_{j})=c_{j}B_{Q}(I_{j})
\]
and%
%
\begin{equation}\label{bs}
B_{c,Q}(I_{i_{1}}\wedge\cdots\wedge I_{i_{l}})=\Upsilon
_{h=1}^{l}B_{c,Q}(I_{i_{h}}),
\end{equation}
where ``$\Upsilon$'' refers to element by element multiplication.
Let the row vector in $T_{c}(Q)$ corresponding to $I_{i_{1}}\wedge
\cdots\wedge I_{i_{l}}$ be $B_{c,Q}(I_{i_{1}}\wedge\cdots\wedge
I_{i_{l}})$.

For instance, with $c=(c_1,c_2,c_3)$, the $T_c(Q)$ corresponding to
the $T$-matrix in \eqref{NewT} would be
%
\begin{equation}\label{AsEx}
T_{c}(Q)= \pmatrix{
c_{1} & 0 & c_{1} \cr
0 & c_{2} & c_{2} \cr
0 & 0 & c_{3} \cr
0 & 0 & c_{1}c_{2}%
} .
\end{equation}

Lastly, we consider the situation that both the probability of
making a mistake and the probability of guessing correctly could be
strictly positive. By this, we mean that the probability that a
subject responds positively to item $i$ is $c_{i}$ if s/he is
capable of doing so; otherwise the probability is $g_{i} $. We
create a corresponding $T_{c,g}(Q)$ by slightly modifying
$T_{c}(Q)$. We define row vector
\[
\mathbb{E}=(1,\ldots,1).
\]
When there is no ambiguity, we omit the length index of
$\mathbb{E}$.
Now, let%
\[
B_{c,g,Q}(I_{i})=g_{i}\mathbb{E}+(c_{i}-g_{i})B_{Q}(I_{i})
\]
and%
%
\begin{equation}\label{bsg}
B_{c,g,Q}(I_{i_{1}}\wedge\cdots\wedge I_{i_{l}})=\Upsilon
_{h=1}^{l}B_{c,g,Q}(I_{i_{h}}).
\end{equation}
Let the row vector in $T_{c,g}(Q)$ corresponding to $I_{i_{1}}\wedge
\cdots\wedge I_{i_{l}}$ be $B_{c,g,Q}(I_{i_{1}}\wedge\cdots\wedge
I_{i_{l}})$. For instance, the matrix $T_{c,g}$ corresponding to the
$T_{c}(Q)$ in \eqref{AsEx} is
%
\begin{equation}\label{AsgEx}
T_{c,g} (Q)= \pmatrix{
c_{1} & g_{1} & c_{1} \cr
g_{2} & c_{2} & c_{2} \cr
g_{3} & g_{3} & c_{3} \cr
c_{1}g_{2} & g_{1}c_{2} & c_{1}c_{2}%
} .
\end{equation}
%




\subsection{Estimation of the $Q$-matrix and consistency results}

Having concluded our preparations, we are now ready to introduce our
estimators for $Q$. Given $c$ and $g$, we define
%
\begin{equation}\label{Score}S_{c,g}(Q)=\inf_{\mathbf p'\in
[0,1]^{2^k-1}} |T_{c,g}(Q)\mathbf p'+p_{\mathbf0}'\mathbf g-\alpha|,
\end{equation}
where $\mathbf p'=(p'_{\mathbf A}\dvtx  \mathbf A \in\{0,1\}^k \setminus
\{\mathbf 0\})$, $p'_{\mathbf0 }= p'_{0\ldots0}$ and
%
\begin{equation}
\label{g}
\mathbf{g}= \pmatrix{
g_{1} \cr
\vdots \cr
g_{k} \cr
g_{1}g_{2} \cr
\vdots \cr
g_{k-1}g_{k} \cr
g_{1}g_{2}g_{3} \cr
\vdots
}
\begin{array}{l}
I_{1} \\
\vdots \\
I_{k} \\
I_{1}\wedge I_{2} \\
\vdots \\
I_{k-1}\wedge I_{k} \\
I_{1}\wedge I_{2}\wedge I_{3} \\
\vdots
\end{array}
\end{equation}
The labels to the right of the vector indicate the corresponding row
vectors in $T_{c,g}(Q)$. The minimization in \eqref{Score} is
subject to constraints that
\[
p'_{\mathbf A} \in[0,1]
\quad\mbox{and}\quad \sum_{\mathbf A}p_{\mathbf A}'=1.
\]
The vector $\mathbf
g$ contains the probabilities of providing positive responses to
items simply by guessing. We propose an estimator of the $Q$-matrix
through a minimization problem, that is,
%
\begin{equation}\label{Est}
\hat Q(c,g) = \arg\inf_{Q'} S_{c,g}(Q').
\end{equation}
We write $c$ and $g$ in the argument to emphasize that the estimator
depends on $c$ and $g$. The computation of the minimization in
\eqref{Score} consists of minimizing a quadratic function subject to
finitely many linear constraints. Therefore, it can be done
efficiently.\looseness=1

\begin{theorem}\label{ThmProb}
Suppose that $c$ and $g$ are known and that conditions \textup{C1}--\textup{C5} are in
force. For subject $r$, the responses are generated independently
such that
%
\begin{equation}\label{RProb}P(R_r^i = 1|\xi^i_r) = c_i^{\xi^i_r}
g_i^{1-\xi^i_r},
\end{equation}
where $\xi^i_r$ is defined as in Theorem~\ref{ThmVan}.
Let $\hat Q (c,g)$ be defined as in \eqref{Est}. If $c_i \neq g_i$ for
all $i$
and $T_{c-g}(Q)\pp^{*}$ does not have zero elements, then
\[
\lim_{N\rightarrow\infty}P\bigl(\hat Q(c,g) \sim Q \bigr)=1.
\]
%
Furthermore, let
\[
\tilde{\mathbf p}(c,g) = \arg\inf_{\mathbf p} |T_{c,g}(\hat
Q(c,g))\mathbf p+ p_{\mathbf0} \mathbf g-\alpha|^2,
\]
subject to constraint that $\sum_{\mathbf A} p_{\mathbf A}=1$. Then,
with an appropriate rearrangement of the columns of $\hat Q$, for
any $\varepsilon>0$,
\[
\lim_{N\rightarrow\infty}P\bigl(|\tilde{\mathbf p}(c,g)-\mathbf p^*|\leq
\varepsilon\bigr)=1.
\]
%
\end{theorem}

\begin{remark}\label{RemDist}
There are various metrics one can employ to measure the distance
between the vectors $T_{c,g}(\hat Q(c,g))\mathbf p+ p_{\mathbf0}
\mathbf g$ and $\alpha$. In fact, any metric that generates the same
topology as the Euclidian metric is sufficient to obtain the
consistency results in the theorem. For instance, a principled
choice of objective function would be the likelihood with $\mathbf
p$ profiled out. The reason we prefer the Euclidian metric (versus,
for instance, the full likelihood) is that the evaluation of $S(Q)$
is easier than the evaluation based on other metrics. More
specifically, the computation of current $S(Q)$ consists of
quadratic programming types of well oiled optimization techniques.
\end{remark}


\section{Extension to the situation with unknown slipping
probabilities}\label{SecSlip}

In this section, we further extend our results to the situation
where the slipping probabilities are unknown and the guessing
probabilities are known. In the context of standard exams, the
guessing probabilities can typically be set to zero for open
problems. For instance, the chance of guessing the correct answer to
$``(3+2)\times2 =\, $?'' is very small. On the other hand, for
multiple choice problems, the guessing probabilities cannot be
ignored. In that case, $g_i$ can be considered as $1/n$ when there
are $n$ choices; see Remark~\ref{remark4.2} for more description.

\subsection{Estimator of $c$}\label{SecEst}

We provide two estimators of $c$ given $Q$ and $g$. One is
applicable to all $Q$-matrices, but computationally intensive. The
other is computationally easy, but requires certain structures of
$Q$. Then, we combine them into a single estimator.

\subsubsection*{A general estimator}
We first provide an estimator of $c$
that is applicable to all $Q$-matrices. Considering that the estimator
of $Q$ minimizes the objective function $S_{c,g}(Q)$, we propose the
following estimator of~$c$:
%
\begin{equation}\label{EstSGen}
\tilde c(Q,g) = \arg\inf_{c\in[0,1]^m} S_{c,g} (Q).
\end{equation}

\subsubsection*{A moment estimator}

The computation of $\tilde c(Q,g)$ is typically intensive. When the
$Q$-matrix has a certain structure, we are able to estimate $c$
consistently based on estimating equations.

For a particular item $i$, suppose that there exist items
${i_1},\ldots,{i_l}$ (different from $i$) such that
%
\begin{equation}\label{suff}B_Q(I_i \wedge I_{i_1}\wedge\cdots\wedge
I_{i_l})=B_Q(I_{i_1}\wedge\cdots\wedge I_{i_l}),
\end{equation}
that is, the attributes required by item $i$ are a subset of the
attributes required by ${i_1},\ldots,{i_l}$.


Let ${c-g} = (c_1-g_1,\ldots,c_m-g_m)$ and
\[
\tilde T_{c,g}(Q)= \pmatrix{
\mathbf g &T_{c,g}(Q) \cr
1& \mathbb{E}%
}.
\]
We borrow a result which will be given in the proof of
Proposition~\ref{PropSG} (Section~\ref{SecLem}) to say that there
exists a matrix $D$ (only depending on $g$) such that
\[
D\tilde T_{c,g}(Q) =(\mathbf0,T_{{c-g}}(Q)).
\]
Let $a_g$ and $a_{*g}$ be the row vectors in $D$ corresponding to
$I_{i_1}\wedge\cdots\wedge I_{i_i}$ and $I_i\wedge I_{i_1}\wedge
\cdots
\wedge I_{i_i}$ (in $T_{{c-g}}(Q)$).

Then, 
%
\begin{eqnarray}
\label{EQ}
\frac{a_{\ast g}^{\top} {
\alpha \choose
1%
} }{ a_{g}^{\top} {
\alpha \choose
1%
} }&=&\frac{a_{\ast g}^{\top}\tilde{T}_{c,g}(Q) {
p_{\mathbf{0}}^* \choose
\mathbf{p}^*%
} }{ a_{g}^{\top}\tilde{T}_{c,g}(Q) {
p_{\mathbf{0}}^* \choose
\mathbf{p}^*%
} }+\mathrm{o}_{p}(1) \nonumber
\\[-8pt]
\\[-8pt]
&=&\frac{B_{{c-g},Q}(I_{i}\wedge
I_{i_{1}}\wedge\cdots\wedge
I_{i_{l}})\mathbf{p}^*}{ B_{{c-g},Q}(I_{i_{1}}\wedge\cdots\wedge
I_{i_{l}})%
\mathbf{p}^*}+\mathrm{o}_{p}(1)\mathop{\rightarrow}^{p}(c_{i}-g_{i}),\nonumber
\end{eqnarray}
where the vectors $a_g$ and $a_{*g}$ only depend on $g$.

Therefore, the corresponding estimator of $c_i$ would be
%
\begin{equation}\label{Shat}
\bar c_i(Q,g) = g_i+ \frac{a^\top_{*g} {
\alpha\choose
1%
} }{a^\top_{g} {
\alpha\choose
1%
}} .
\end{equation}
Note that the computation of $\bar c_i(Q,g)$ only consists of affine
transformations and therefore is very fast.

\begin{proposition}\label{PropSHat}
Suppose conditions \textup{C3}, \eqref{RProb} and \eqref{suff} are true.
Then $\bar c_i \rightarrow c_i$ in probability as $N\rightarrow
\infty$.
\end{proposition}

\begin{pf}
By the law of large numbers,
\[
a_{\ast g}^{\top} \pmatrix{
\alpha \cr
1%
} - a_{\ast g}^{\top}\tilde{T}_{c,g}(Q) \pmatrix{
p_{\mathbf{0}}^* \vspace*{2pt}\cr
\mathbf{p}^*%
} \rightarrow0,\qquad a_{g}^{\top} \pmatrix{
\alpha \cr
1%
} - a_{ g}^{\top}\tilde{T}_{c,g}(Q) \pmatrix{
p_{\mathbf{0}}^* \vspace*{2pt}\cr
\mathbf{p}^*%
} \rightarrow0,
\]
in probability as $N\rightarrow\infty$. By
the construction of $a_{*g}$ and $a_{g}$, we have
\begin{eqnarray*}a_{\ast g}^{\top}\tilde{T}_{c,g}(Q) \pmatrix{
p_{\mathbf{0}}^* \vspace*{2pt}\cr
\mathbf{p}^*%
} &=& B_{{c-g},Q}(I_{i}\wedge I_{i_{1}}\wedge\cdots\wedge
I_{i_{l}})\mathbf{p}^*,\\
a_{g}^{\top}\tilde{T}_{c,g}(Q) \pmatrix{
p_{\mathbf{0}}^* \vspace*{2pt}\cr
\mathbf{p}^*%
} &=& B_{{c-g},Q}(I_{i_{1}}\wedge\cdots\wedge
I_{i_{l}})\mathbf{p}^*.
\end{eqnarray*}
Thanks to \eqref{suff}, we
have
\[
\frac{a_{\ast g}^{\top} {
\alpha \choose
1%
} }{ a_{g}^{\top} {
\alpha \choose
1%
}} \rightarrow c_i - g_i.
\]
\upqed\end{pf}

\subsubsection*{Combined estimator}
Lastly, we combine $\bar c_i$ and
$\tilde c_i$.
For each $Q$, we write $c=(c^*,c^{**})$. For each $c_i$ in the
sub-vector $c^*$, \eqref{suff} holds. Let $\bar c^*(Q,g)$ be defined
in \eqref{Shat} (element by element). For $c^{**}$, we let $\tilde
c^{**} (Q,g) = \arg\inf_{c^{**}} S_{(\bar c^*(Q,g),c^{**}),g} (Q)$.
Finally, let $\hat c(Q,g)= (\bar c^*(Q,g), \tilde c^{**} (Q,g))$.
Furthermore, each element of $\hat c(Q,g)$ greater than one is set
to be one and each element less than zero is set to be zero.
Equivalently, we impose the constraint that $\hat c(Q,g)\in
[0,1]^m$.

\subsection{Consistency result}

\begin{theorem}\label{CorProb}
Suppose that $g$ is known and the conditions in Theorem~\ref{ThmProb} hold. Let
\[
\hat Q_{\hat c}(g) = \arg\inf_{Q'} S_{\hat c (Q',g),g}(Q'),\qquad
\tilde{\mathbf p}_{\hat c}(g) = \arg\inf_{\mathbf p} \bigl|T_{\hat
c(\hat Q,g),g}(\hat Q_{\hat c}(g))\mathbf p+p_{\mathbf0}\mathbf
g-\alpha\bigr|.
\]
The second optimization is subject to constraint that
$\sum_{\mathbf A} p_{\mathbf A}=1$. Then,
\[
\lim_{N\rightarrow\infty}P \bigl(\hat Q_{\hat c}(g)\sim Q \bigr)=1.
\]
Furthermore, if the estimator $\tilde c(Q,g)$, defined in
\eqref{EstSGen}, is consistent, then by appropriately rearranging
the columns of $\hat Q_{\hat c}(g)$, for any $\varepsilon>0$,
\[
\lim_{N\rightarrow\infty}P \bigl(|\tilde{ \mathbf p}_{\hat c}(g)-\mathbf
p^*|\leq
\varepsilon\bigr)=1.
\]
\end{theorem}

\begin{remark}
The consistency of $\hat Q_{\hat c}(g)$ does not rely on the
consistency of $\tilde c(Q,g)$, which is mainly because of the
central intermediate result in Proposition~\ref{PropSG}. The
consistency of $\tilde c(Q,g)$ is a necessary condition for the
consistency of $\tilde{\mathbf p}_{\hat c}(g)$.

For most usual situations, $(\mathbf p^*, c)$ is estimable based on
the data given a correctly specified $Q$-matrix. Nonetheless, there
are some rare occasions in which nonidentifiability does exist. We
provide one example, explained at the intuitive level, to illustrate
that it is not always possible to consistently estimate $c$ and
$\mathbf p^*$. This example is simply to justify that the existence
of the consistent estimator for $c$ in the above theorem is not an
empty assumption. Consider a complete matrix $Q= \mathcal I_k$. The
degrees of freedom of a $k$-way binary table is $2^k-1$. On the
other hand, the dimension of parameters $(\mathbf p^*, c)$ is $2^k-1
+k$. Therefore, $\mathbf p^*$ and $c$ cannot be consistently
identified without additional information. This problem is typically
tackled by introducing addition parametric assumptions such as
$\mathbf p^*$ satisfying certain functional form or in the Bayesian
setting (weakly) informative prior distributions \cite{Gelman08}.
Given that the emphasis of this paper is the inference of
$Q$-matrix, we do not further investigate the identifiability of
$(\mathbf p^*, c)$. Nonetheless, estimation for $(\mathbf p^*, c)$
is definitely an important issue.
\end{remark}

\begin{remark}\label{remark4.2}
Assuming that the guessing probability $g_i$ being known is somewhat
strong. For complicated situations, such as for multiple choice
problems the incorrect choices do not look ``equally incorrect'',
the guessing probability is typically not $1/n$.
In Theorem~\ref{CorProb}, we
make this assumption mostly for technical reasons.

One can certainly provide the same treatment to the unknown guessing
probabilities just as to the slipping probabilities by plugging in a
consistent estimator of $g_i$ or profiling it out (like $\tilde c$).
However, the rigorous establishment of the consistency results is
certainly much more difficult and additional technical conditions may
be needed. We leave the analysis of the problem with unknown guessing
probability
to the future study.
\end{remark}
%

\section{Discussion}\label{SecDis}

This paper provides basic theoretical results of the estimation of
$Q$-matrix, a key element in modern cognitive diagnosis. Under the
conjunctive model assumption, sufficient conditions are developed
for the $Q$-matrix to be identifiable up to an equivalence relation
and the corresponding consistent estimators are constructed. The
equivalence relation defines a natural partition of the space of
$Q$-matrices and may be viewed as the finest ``resolution'' that is
possibly distinguishable based on the data, unless there is
additional information about the specific meaning of each attribute.
Our results provide the first steps for statistical inference about
$Q$-matrices by explicitly specifying the conditions under which two
$Q$-matrices lead to different response distributions. We believe
that these results, especially the intermediate results in Section~\ref{SecProof}, can also be applied to general conjunctive models.

There are several directions along which further exploration may be
pursued. First, some conditions may be modified to reflect practical
circumstance. For instance, if the population is not fully
diversified, meaning that certain attribute profiles may never
exist, then condition C4 cannot be expected to hold. To ensure
identifiability, we will need to impose certain structures on the
$Q$-matrix. In the addition-multiplication example of Section~\ref{SecEx}, if individuals capable of multiplication are also
capable of addition, then we may need to impose the natural
constraint that every item that requires multiplication should also
require addition, which also implies that the $Q$-matrix is never
complete.

Second, when an a priori ``expert's'' knowledge of the
$Q$-matrix is available, we may wish to incorporate such information
into the estimation. This could be in the form of an additive
penalty function attached to the objective function $S$. Such
information, if correct, not only improves estimation accuracy but
also reduces the computational complexity -- one can just perform a
minimization of $S(Q)$ in a neighborhood around the expert's
$Q$-matrix.

Third, throughout this paper we assume that the number of attributes
(dimension) is known. In practice, it would be desirable to develop
a data driven way to estimate the dimension, not only to deal with
the situation of unknown dimension, but also to check if the assumed
dimension is correct. One possible way to tackle the problem is to
introduce a penalty function similar to that of BIC \cite{Sch78}
which would give a consistent estimator of the $Q$-matrix even if
the dimension is unknown.

Fourth, one issue of both theoretical and practical importance is
the inference of the parameters additional to the $Q$-matrix, such
as the slipping ($s=1-c$), guessing ($g$) parameters and the
attribute distribution $\mathbf p^*$. In the current paper, given
that the main interesting parameter is the $Q$-matrix, the
estimations of $\mathbf p^*$ and $c$ are treated as by-product of
the main results. On the other hand, given a known $Q$, the
identifiability and estimation of these parameters are important
topics. In the previous discussion, we provided a few examples for
potential identifiability issues. Further careful investigation is
definitely of great importance and challenges.

Fifth, the rate of convergence of the estimator $\hat Q$
is not only of theoretical importance. From a practical point of few,
it is crucial to study the rate of convergence as the scale of the
problem becomes large in terms of the number of attributes and number of
items.

Lastly, the optimization of $S(Q)$ over the space of $m\times k$
binary matrices is a nontrivial problem. It consists of evaluating
the function $S$ $2^{m\times k}$ times. This is a substantial
computational load if $m$ and $k$ are reasonably large. As mentioned
previously, this computation might be reduced by additional
information about the $Q$-matrix or splitting the $Q$-matrix into
small sub-matrices. Nevertheless, it would be highly desirable to
explore the structures of the $Q$-matrix and the function $S$ so as
to compute $\hat Q$ more efficiently.

\section{Proofs of the theorems}\label{SecProof}

\subsection{Several propositions and lemmas}\label{SecLem}
To make the discussion smooth, we postpone several long proofs to the
\hyperref[app]{Appendix}.

\begin{proposition}\label{PropRank}
Suppose that $Q$ is complete and matrix $T(Q)$ is saturated. Then,
we are able to arrange the columns and rows of $Q$ and $T(Q)$ such
that $T(Q)_{1:(2^k-1)}$ has full rank and $T(Q)$ has full column
rank.
\end{proposition}

\begin{pf}
Provided that $Q$ is complete, without loss of generality we assume
that the $i$th row vector of $Q$ is $e_i^\top$ for $i=1,\ldots,k$,
that is, item $i$ only requires attribute $i$ for each $i=1,\ldots,k$.
Let the first $2^k -1$ rows of $T(Q)$ be associated with
$\{I_1,\ldots,I_k\}$. In particular, we let the first $k$ rows
correspond to $I_1,\ldots,I_k$ and the first $k$ columns of $T(Q)$
correspond to $\mathbf A$'s that only have one attribute. We
further arrange the next $C^k_2$ 
rows of $T(Q)$ to correspond to combinations of two items, $I_i
\wedge I_j$, $i\neq j$. The next $C^k_2$ columns of $T(Q)$
correspond to $\mathbf A$'s that only have two positive attributes.
Similarly, we arrange $T(Q)$ for combinations of three, four, and up
to $k$ items. Therefore, the first $2^{k}-1$ rows of $T(Q)$ admit a
block upper triangle form. In addition, we are able to further
arrange the columns within each block such that the diagonal
matrices are identities, so that $T(Q)$ has form
%
\begin{equation}\label{A}
\begin{array}{c}
I_{1},I_{2},\ldots\\
I_{1}\wedge I_{2},I_{1}\wedge I_{3},\ldots\\
I_{1}\wedge I_{2}\wedge I_{3},\ldots\\
\vdots%
\end{array}
\pmatrix{
\mathcal I_{k} & \ast & \ast & \ast & \cdots& \cr
0 & \mathcal I_{C_{2}^{k}} & \ast & \ast & & \cr
0 & 0 & \mathcal I_{C_{3}^{k}} & \ast & & \cr
\vdots& \vdots& \vdots& & &
}.
\end{equation}
Note that $T(Q)$ has $2^k-1$ columns and $T(Q)_{1:(2^k-1)}$ obviously
has full rank, therefore $T(Q)$ has
full column rank.
\end{pf}

From now on, we assume that $Q_{1:k}=\mathcal I_k$ and the first
$2^{k}-1$ rows of $T(Q)$ are arranged in the order as in \eqref{A}.

\begin{proposition}\label{PropRankS}
Suppose that $Q$ is complete, $T(Q)$ is saturated, and $c\ncong
\mathbf0$. Then, $T_{c}(Q)$ and $T_c(Q)_{1:(2^k-1)}$ have full
column rank.
\end{proposition}
\begin{pf}
By Proposition~\ref{PropRank}, \eqref{As} and the fact that $D_c$ is
a diagonal matrix of full rank as long as $c\ncong\mathbf0$,
\[
T_c(Q) = D_c T(Q),
\]
is of full column rank.
\end{pf}

The following two propositions, which compare the column spaces of
$T_c(Q)$ and $T_c(Q')$, are central to the proof of all the
theorems. Their proofs are delayed to the \hyperref[app]{Appendix}.

The first proposition discusses the case where $Q'_{1:k}$ is
complete. We can always rearrange the columns of $Q'$ so that
$Q_{1:k}=Q'_{1:k}$. In addition, according to the proof of
Proposition~\ref{PropRank}, the last column vector of $T_c(Q)$
corresponds to attribute $\mathbf A=(1,\ldots,1)^\top$. Therefore, this
column vector is all of nonzero entries.

\begin{proposition}\label{PropComp}
Assume that $Q$ is a complete matrix and $T(Q)$ is saturated.
Without loss of generality, let $Q_{1:k} =\mathcal I_k$. Assume that
the first $k$ rows of $Q'$ form a complete matrix. Further, assume
that $Q_{1:k}=Q'_{1:k}=\mathcal I_k$. If $Q'\neq Q$ and $c\ncong
\mathbf0$, under the conditions in Theorem~\ref{CorProb},
$T_{c}(Q)\pp^{*}$
is not in the column space $C(T_{c'}(Q'))$ for all
$c'\in\mathbb R^m$.
\end{proposition}

The next proposition discusses the case where $Q'_{1:k}$ is
incomplete.

\begin{proposition}\label{PropIncomp}
Assume that $Q$ is a complete matrix and $T(Q)$ is saturated.
Without loss of generality, let $Q_{1:k} =\mathcal I_k$.
If $c\ncong\mathbf0$ and
$Q'_{1:k}$ is incomplete, under the conditions in Theorem~\ref
{CorProb}, $T_{c}(Q)\pp^{*}$
is not in the column space $C(T_{c'}(Q'))$
for all $c'\in R^m$.
\end{proposition}

The next result is a direct corollary of these two propositions. It
follows by setting $c_i=1$ and $g_i =0$ for all $i=1,\ldots,m$.

\begin{corollary}\label{CorRank}
If $Q\nsim Q'$, under the conditions of Theorem~\ref{CorProb},
$T_{c}(Q)\pp^{*}$
is not in the column space $C(T_{c'}(Q'))$
for all $c'\in[0,1]^m$.
\end{corollary}

To obtain a similar proposition for the cases where the $g_{i}$'s
are nonzero, we will need to expand the $T_{c,g}(Q)$ as follows. As
previously defined, let
%
\begin{equation}\label{tildeT}
\tilde{T}_{c,g}(Q)= \pmatrix{
\mathbf{g} & T_{c,g}(Q) \cr
1 & \mathbb{E}%
} .
\end{equation}
The last row of $\tilde{T}_{c,g}(Q)$ consists entirely of ones.
Vector $\mathbf g$ is defined as in \eqref{g}.

\begin{proposition}
\label{PropSG}Suppose that $Q$ is a complete matrix, $Q'\nsim Q$,
$T$ is saturated and $c\ncong g$. Let $\pp^{*}_{0}= (p^{*}_{\mathbf
0},(\pp^{*})^{\top})^{\top}$.
Under the conditions of Theorem~\ref{CorProb}, $\tilde T_{c,g}(Q)\pp^{*}_0$
is not in the column space $C(\tilde T_{c',g}(Q'))$
for all $c'\in[0,1]^m$.
In addition, $\tilde T_{c,g}(Q)$ is of full column rank.
\end{proposition}

To prove Proposition~\ref{PropSG}, we will need the following lemma.

\begin{lemma}
\label{LemColT}Consider two matrices $T_{1}$ and $T_{2}$ of the same
dimension. If $T_{1}\pp\in C(T_{2})$, then for any matrix $D$
of appropriate dimension for multiplication, we have%
\[
DT_{1} \pp\in C(DT_{2}).
\]

Conversely, if for some $D$, $DT_{1}\pp$ does not belong
to $C(DT_{2})$, then $T_{1}\pp$ does not belong to $
C(T_{2})$.
\end{lemma}

\begin{pf}
Note that $DT_i$ is just a linear row transform of $T_i$ for
$i=1,2$. The conclusion is immediate by basic linear algebra.
\end{pf}

\begin{pf*}{Proof of Proposition~\ref{PropSG}}
Thanks to Lemma~\ref{LemColT}, we only need to find a matrix $D$
such that
$D\tilde{T}_{c,g}(Q) \pp^{*}_{0}$ does not belong to the column space
of $D\tilde{T}_{c^{\prime},g}(Q^{\prime})$ for all $c^{\prime}\in
\lbrack0,1]^{m}$.

We define
\begin{eqnarray*}
{c-g}&=&(c_{1}-g_{1},\ldots,c_{m}-g_{m}),
\\
 c^{\prime}-{g}&=&(c_{1}^{\prime
}-g_{1},\ldots,c_{m}^{\prime}-g_{m}).
\end{eqnarray*}
We claim that there exists a matrix $D$ such that%
\[
D\tilde{T}_{c,g}(Q)= (
0,  T_{{c-g}}(Q)%
)
\]
and
\[
D\tilde{T}_{c^{\prime},g}(Q^{\prime})= (
0,  T_{c^{\prime}-{g}}(Q^{\prime})%
) ,
\]
where the choice of $D$ does not depends on $c$ or $c^{\prime}$. In
the rest of the proof, we construct such a $D$-matrix for
$\tilde{T}_{c,g}(Q)$. The verification for $\tilde{T}_{c^{\prime
},g}(Q^{\prime})$ is completely analogous. Note that each row in
$D\tilde{T}_{c,g}(Q)$ is just a linear combination of rows of
$\tilde{T}_{c,g}(Q)$. Therefore, it suffices
to show that every row vector of the form%
\[
\bigl(0,B_{{c-g},Q}(I_{i_{1}}\wedge\cdots\wedge I_{i_{l}})\bigr)
\]
can be written as a linear combination of the row vectors of
$\tilde{T} _{c,g}(Q)$. We prove this by induction. First note that
for each $1\leq i\leq m$,
%
\begin{equation}
\label{single}
(0,B_{{c-g},Q}(I_{i}))=(c_{i}-g_{i})(0,B_{Q}(I_{i}))=(g_{i},B_{c,g,Q}(I_{i}))-g_{i}\mathbb
{E}.
\end{equation}
Suppose that all rows of the form
\[
\bigl(0,B_{{c-g},Q}(I_{i_{1}}\wedge\cdots\wedge I_{i_{l}})\bigr)
\]
for all $1\leq l\leq j$ can be written as linear combinations of the
row vectors of $\tilde{T}_{c,g}(Q)$ with coefficients only depending
on $g_{1},\ldots,g_{m}$. Thanks to \eqref{single}, the case of $j=1$
holds. Suppose the statement holds for some general $j$. We consider
the case of $j+1$. By definition,
%
\begin{eqnarray}\label{BSG}
\bigl(g_{i_1}\ldots g_{i_{j+1}}, B_{c,g,Q}(I_{i_{1}}\wedge\cdots\wedge
I_{i_{j+1}}) \bigr)&=&\Upsilon
_{h=1}^{j+1} (g_{i_h},B_{c,g,Q}(I_{i_{h}}) )\nonumber
\\[-8pt]
\\[-8pt]
&=&\Upsilon_{h=1}^{j+1} \bigl(g_{i_h}\mathbb
E+(0,B_{{c-g},Q}(I_{i_h})) \bigr). \nonumber
\end{eqnarray}
Let ``$\ast$'' denote element-by-element multiplication. For every
generic vector $V'$ of appropriate length,
\[
\mathbb{E}\ast V'=V'.
\]
We expand the right-hand side of \eqref{BSG}. The last term would be
\[
\bigl( 0, B_{{c-g},Q}(I_{i_{1}}\wedge\cdots\wedge
I_{i_{j+1}}) \bigr)=\Upsilon_{h=1}^{j+1} ( 0,
B_{{c-g},Q}(I_{i_{h}}) ).
\]
From the induction assumption and definition \eqref{bs}, the other
terms on both sides of \eqref{BSG} belong to the row space of
$\tilde{T}_{c,g}(Q)$. Therefore, $( 0, B_{{c-g},Q}(I_{i_{1}}\wedge
\cdots\wedge I_{i_{j+1}}))$ is also in the row space of
$\tilde{T}_{c,g}(Q)$. In addition, all the corresponding
coefficients only consist of $g_{i}$. Therefore, one can construct a
$(2^m-1) \times2^m$ matrix $D$ such that
\[
D\tilde{T}_{c,g}(Q)= (
0,  T_{{c-g}}(Q)%
) .
\]
Because $D$ is free of $c$ and $Q$, we have%
\[
D\tilde{T}_{c^{\prime},g}(Q^{\prime})= (
0,  T_{c^{\prime}-g}(Q^{\prime})%
) .
\]
In addition, thanks to Propositions~\ref{PropComp} and~\ref{PropIncomp}, $D\tilde T_{c,g}(Q)\pp^{*}_0=T_{c-g}(Q)\pp^{*}$
is not in the column space $C(T_{c'-g}(Q'))=C(D\tilde T_{c',g}(Q'))$
for all $c'\in[0,1]^m$. Therefore, by Lemma~\ref{LemColT},
$\tilde T_{c,g}(Q)\pp^{*}_0$
is not in the column space $C(\tilde T_{c',g}(Q'))$
for all $c'\in[0,1]^m$.

In addition,
\[
\pmatrix{
D \cr
e_{2^{m}}^\top%
} \tilde{T}_{c,g}(Q)
\]
is of full column rank, where $e^\top_{2^m}$ is a $2^m$ dimension
row vector with last element being one and rest being zero.
Therefore, $\tilde{T}_{c,g}(Q)$ is also of full column rank.
\end{pf*}

\subsection{Proof of the theorems}

Using the results of the previous propositions and lemmas, we now
proceed to prove our theorems.

\begin{pf*}{Proof of Theorem~\ref{ThmVan}}
Consider $Q'\nsim Q$ and $T(\cdot)$ saturated.
Recall that $\hat{\mathbf p}$ is the
vector containing $\hat p_{\mathbf A}$'s with $\mathbf A \ncong
\mathbf0$, where
\[
\hat p_{\mathbf A} = \frac1 N\sum_{r=1}^N \mathbf1(\mathbf A_r
=\mathbf A).
\]
For any $\mathbf p^*\succ\mathbf0$, since $\hat{\mathbf
p}\rightarrow\mathbf p^*$ almost surely, according to
Corollary~\ref{CorRank}, $\alpha= T(Q)\hat{\mathbf
p}$ by \eqref{id}, and $T(Q)\mathbf p^* \notin C(T(Q'))$, there exists
$ \delta
>0$ such that,
\[
\lim_{N\rightarrow\infty}P \Bigl(\inf_{\mathbf p\in[0,1]^{2^k-1}}
|T(Q')\mathbf p - \alpha| > \delta \Bigr) =1
\]
and
\[
P \Bigl(\inf_{\mathbf p\in[0,1]^{2^k-1}} |T(Q)\mathbf p - \alpha|=0 \Bigr)=1.
\]
Given that there are finitely many $m\times k$ binary matrices,
$P(\hat Q\sim Q) \rightarrow1$ as $N\rightarrow\infty$. In fact,
we can arrange the columns of $\hat Q$ such that $P(\hat Q =
Q)\rightarrow1$ as $N\rightarrow\infty$.

Note that $\hat{\mathbf p}$ satisfies the identity
\[
T(Q)\hat{\mathbf p} = \alpha.
\]
In addition, since $T(Q)$ is of full rank (Proposition~\ref{PropRank}), the solution to the above linear equation is
unique. Therefore, the solution to the optimization problem
$\inf_{\mathbf p} |T(Q)\mathbf p - \alpha|$ is unique and is~$\hat{\mathbf p}$. Notice that when $\hat Q =Q$, $\tilde{\mathbf p}
=\arg\inf_{\mathbf p} |T(\hat Q) \mathbf p - \alpha| = \hat
{\mathbf p}$. Therefore,
\[
\lim_{N\rightarrow\infty}P(\tilde{\mathbf p} = \hat{\mathbf p} )=1.
\]
Together with the consistency of $\hat{\mathbf p}$, the conclusion
of the theorem follows immediately.
\end{pf*}

\begin{pf*}{Proof of Theorem~\ref{ThmProb}}
Note that for all $Q'$%
\[
T_{c,g}(Q')\mathbf p+p_{\mathbf{0}}\mathbf{g}-\alpha=(\mathbf
g,T_{c,g}(Q' )) \pmatrix{
p_{\mathbf{0}}  \cr
\mathbf p%
} -\alpha.
\]
By the law of large numbers,
\[
|T_{c,g}(Q)\mathbf p^{\ast}+p_{\mathbf{0}}^{\ast}\mathbf{g}-\alpha
|= \biggl\vert(\mathbf g,T_{c,g}(Q)) \pmatrix{
p_{\mathbf{0}}^{\ast} \vspace*{2pt}\cr
\mathbf p^{\ast}%
} -\alpha \biggr\vert\rightarrow0
\]
almost surely as $N\rightarrow\infty$. Therefore,
\[
S_{c,g}(Q)\rightarrow0
\]
almost surely as $N\rightarrow\infty$.

For any $Q^{\prime}\nsim Q$, note that
\[
\pmatrix{
\alpha \cr
1%
} \rightarrow\tilde{T}_{c,g}(Q) \pmatrix{
p_{\mathbf{0}}^{\ast} \vspace*{2pt}\cr
\mathbf{p}^{\ast}%
} .
\]
According to Proposition~\ref%
{PropSG} and the fact that $\mathbf{p}^{\ast}\succ\mathbf{0}$,
there exists $\delta(c^{\prime})>0$ such that $\delta(c^{\prime
})$ is continuous in $c^{\prime}$ and
\[
\inf_{\mathbf{p},p_{\mathbf{0}}} \biggl\vert\tilde{T}_{c^{\prime
},g}(Q^{\prime}) \pmatrix{
p_{\mathbf{0}} \cr
\mathbf{p}%
} -\tilde{T}_{c,g}(Q) \pmatrix{
p_{\mathbf{0}}^{\ast} \vspace*{2pt}\cr
\mathbf{p}^{\ast}%
} \biggr\vert>\delta(c^{\prime}).
\]
By elementary calculus,
\[
\delta\triangleq\inf_{c^{\prime}\in\lbrack0,1]^{m}}\delta
(c^{\prime})>0
\]
and%
\[
\inf_{c^{\prime},\mathbf{p},p_{\mathbf{0}}} \biggl\vert
\tilde{T}_{c^{\prime},g}(Q^{\prime}) \pmatrix{
p_{\mathbf{0}} \cr
\mathbf{p}%
} -\tilde{T}_{c,g}(Q) \pmatrix{
p_{\mathbf{0}}^{\ast} \vspace*{2pt}\cr
\mathbf{p}^{\ast}%
} \biggr\vert>\delta.
\]
Therefore,
\[
P \biggl( \inf_{c^{\prime},\mathbf{p},p_{\mathbf{0}}} \biggl\vert\tilde{T}%
_{c^{\prime},g}(Q^{\prime}) \pmatrix{
p_{\mathbf{0}} \cr
\mathbf{p}%
} - \pmatrix{
\alpha \cr
1%
} \biggr\vert>\delta/2 \biggr) \rightarrow1,
\]
as $N\rightarrow\infty$. For the same $\delta$, we have%
\[
P \biggl( \inf_{c^{\prime},\mathbf{p},p_{\mathbf{0}}} \biggl\vert(\mathbf{g}%
,T_{c^{\prime},g}(Q^{\prime})) \pmatrix{
p_{\mathbf{0}} \cr
\mathbf{p}%
} -\alpha \biggr\vert>\delta/2 \biggr)=
P\Bigl(\inf_{c'}S_{c',g}(Q')> \delta/2\Bigr) \rightarrow1.
\]
The above minimization on the left of the equation is subject to the
constraint that
\[
\sum_{\mathbf{A}\in\{0,1\}^{k}}p_{\mathbf{A}}=1.
\]
Together with the fact that there are only finitely many $m\times k$
binary matrices, we have
\[
P\bigl(\hat{Q}(c,g)\sim Q\bigr)=1.
\]
We arrange the columns of $\hat{Q}(c,g)$ so that $P(\hat{Q}%
(c,g)=Q)\rightarrow1$ as $N\rightarrow\infty$.

Now we proceed to the proof of consistency for
$\tilde{\mathbf{p}}(c,g)$. Note
that%
\begin{eqnarray*}
\biggl\vert\tilde{T}_{c,g}(\hat Q(c,g)) \pmatrix{
\tilde{p}_{\mathbf{0}}(c,g) \cr
\tilde{\mathbf{p}}(c,g)%
} - \pmatrix{
\alpha\cr
1%
} \biggr\vert
&\mathop{\rightarrow}\limits^{p}&0,
\\
\biggl\vert\tilde{T}_{c,g}(Q) \pmatrix{
p_{\mathbf{0}}^{\ast} \vspace*{2pt}\cr
\mathbf p^{\ast}%
} - \pmatrix{
\alpha\cr
1%
} \biggr\vert
&\mathop{\rightarrow}\limits^{p}&0.
\end{eqnarray*}
Since $\tilde{T}_{c,g}(Q)$ is a full column rank matrix and $P(\hat
Q(c,g)=Q)\rightarrow1$, $\tilde{\mathbf p}(c,g)\rightarrow\mathbf
p^{\ast}$ in probability.
\end{pf*}

\begin{pf*}{Proof of Theorem~\ref{CorProb}}
Assuming $g$ is known, note that
\[
\inf_{p_{\mathbf{0}},\mathbf p} \biggl\vert\tilde T_{c,g}(Q) \pmatrix{
p_{\mathbf{0}} \cr
\mathbf p
} - \pmatrix{
\alpha \cr
1%
} \biggr\vert
\]
is a continuous function of $c$. According to the results of
Proposition~\ref{PropSHat}, the definition in \eqref{EstSGen}, and
the definition of $\hat c$ in Section~\ref{SecEst}, we obtain that
\[
\inf_{p_{\mathbf{0}},\mathbf p} \biggl\vert\tilde T_{\hat{c}
(Q,g),g}(Q) \pmatrix{
p_{\mathbf{0}} \cr
\mathbf p
} - \pmatrix{
\alpha \cr
1%
} \biggr\vert\rightarrow0,
\]
in probability as $N\rightarrow\infty$. In addition, thanks to
Proposition~\ref{PropSG} and with a similar argument as in the proof
of Theorem~\ref{ThmProb}, $\hat{Q}_{\hat c}(g)$ is a consistent
estimator.

Furthermore, if $\tilde c (Q,g)$ is a consistent estimator, then
$\hat c(Q,g)$ is also consistent. Then, the consistency of
$\tilde{\mathbf p}_{\hat c}(g)$ follows from the facts that $\hat
Q_{\hat c}(g)$ is consistent and $\tilde T_{\hat c,g} (Q)$ is of
full column rank.
\end{pf*}

\begin{appendix}\label{app}
\section*{Appendix: Technical proofs}\vspace*{-12pt}\label{SecTech}
\begin{pf*}{Proof of Proposition~\ref{PropComp}}
Note that $Q_{1:k}=Q_{1:k}^{\prime}=\mathcal{I}_{k}$. Let $T(\cdot
)$ be arranged as in \eqref{A}. Then,
$T(Q)_{1:(2^{k}-1)}=T(Q^{\prime})_{1:(2^{k}-1)}$. Given that $Q\neq
Q^{\prime}$, we have $T(Q)\neq T(Q^{\prime})$. We assume that
$T(Q)_{li}\neq T(Q')_{li}$, where $T(Q)_{li}$ is the entry in the
$l$th row and $i$th column. Since $ T(Q)_{1:(2^{k}-1)}=T(Q^{\prime
})_{1:(2^{k}-1)}$, it is necessary that $l\geq2^{k}$.

Suppose that the $l$th row of the $T(Q^{\prime})$ corresponds to
an item
that requires attributes $i_{1},\ldots,i_{l^{\prime}}$. Then, we
consider $%
1\leq h\leq2^{k}-1$, such that the $h$th row of $T(Q^{\prime})$ is $%
B_{Q^{\prime}}(I_{i_{1}}\wedge\cdots\wedge I_{i_{l^{\prime}}})$.
Then, the $h $th row vector and the $l$th row vector of
$T(Q^{\prime})$ are identical.

Since $T(Q)_{1:(2^{k}-1)}=T(Q^{\prime})_{1:(2^{k}-1)}$, we have $%
T(Q)_{hj}=T(Q^{\prime})_{hj}=T(Q^{\prime})_{lj}$ for
$j=1,\ldots,\allowbreak 2^{k}-1$. If $T(Q)_{li}=0$ and $T(Q^{\prime})_{li}=1$, the
matrices $T(Q)$ and $T(Q^{\prime})$ look like
\begin{eqnarray*}
&&%
\begin{array}{ccccccccccccccccccc}
&&&&&&&&&&&&&&&&&&\begin{array}{ccccc}
&&&&\mbox{column }i \\
&&&&\downarrow \quad
\end{array}
\end{array}
\\
T(Q^{\prime }) &=&%
\begin{array}{c}
\\
\\
\mbox{row }h\rightarrow
\\
\\
\\[3pt]
\mbox{row }l\rightarrow
\end{array}%
\left(
\begin{array}{c@{\quad}c@{\quad}c@{\quad}c@{\quad}c}
\mathcal{I} & \ast  & \ldots  & \ast  & \ldots  \\
\vdots  & \vdots  &  & \ldots  & \ldots  \\
\vdots  & \vdots  & \mathcal{I} & \ldots  & \ldots  \\
\vdots  & \vdots  & \vdots  &  &  \\
\ast  & 1 & \ast  &  &  \\
\ast  & \ast  & \ast  &  &
\end{array}%
\right)
\end{eqnarray*}%
and
\begin{eqnarray*}
&&%
\begin{array}{ccccccccccccccccccc}
&&&&&&&&&&&&&&&&&&\begin{array}{ccccc}
&&&&\mbox{column }i \\
&&&&\downarrow \quad
\end{array}
\end{array}
\\
T(Q) &=&%
\begin{array}{c}
\\
\\
\mbox{row }h\rightarrow
\\
\\
\\[3pt]
\mbox{row }l\rightarrow  \\
\end{array}%
\left(
\begin{array}{c@{\quad}c@{\quad}c@{\quad}c@{\quad}c}
\mathcal{I} & \ast  & \ldots  & \ast  & \ldots  \\
\vdots  & \vdots  &  & \ldots  & \ldots  \\
\vdots  & \vdots  & \mathcal{I} & \ldots  & \ldots  \\
\vdots  & \vdots  & \vdots  &  &  \\
\ast  & 0 & \ast  &  &  \\
\ast  & \ast  & \ast  &  &
\end{array}%
\right) .
\end{eqnarray*}
\begin{longlist}[Case 2]
\item[Case 1.] Either the $h$th or $l$th row vector of $T_{c'}(Q')$
is a zero vector.
The conclusion is immediate because all the entries of $T_{c}(Q)\pp
^{*}$ are nonzero.

\item[Case 2.] The $h$th and $l$th row vectors of $T_{c'}(Q')$ are
nonzero vectors.
Suppose that the $l$th row corresponds to an item. There are three
different situations: according to the true $Q$-matrix (a)~the item in
row $l$ requires strictly more attributes than row $h$, (b) the item in
row $l$ requires strictly fewer attributes than row $h$, (c) otherwise.
We consider these three situations, respectively.
\begin{enumerate}[(b)]
\item[(a)] Under the true $Q$-matrix, there are two types of
sub-populations in consideration: people who are able to answer item(s)
in row $h$ ($p_{1}$) only and people who are able to answer items in
both row $h$ and row $l$ ($p_{2}$). Then, the sub-matrix of $T_{c}(Q)$
and $T_{c'}(Q)$ are like
\begin{center}
\begin{tabular}{l@{\quad}l@{\quad}l}
\multicolumn{3}{c}{\hspace*{26.5pt}$T_{c}(Q)$}  \\
& $p_{1}$ & $p_{2}$   \\
row $h$ & $c_{h}$ & $c_{h}$  \\
row $l$ & $0$ & $c_{l}$  
\end{tabular}
\quad
\begin{tabular}{l@{\quad}l@{\quad}l}
\multicolumn{3}{c}{\hspace*{31.5pt}$T_{c^{\prime}}(Q^{\prime})$}  \\
& $p_{1}$ & $p_{2}$  \\
row $h$ & $c_{h}^{\prime}$ & $c_{h}^{\prime}$  \\[2pt]
row $l$ & $c_{l}^{\prime}$ & $c_{l}^{\prime}$ 
\end{tabular}
\end{center}
We now claim that $c_{l}$ and $c_{l}'$ must be equal (otherwise the
conclusion hold) for the following reason.

Consider the following two rows of $T(Q)$: row A corresponding to the
combination that contains all the items; row B corresponding to the row
that contains all the items except for the one in row $l$.

Rows A and B are in fact identical in $T(Q)$. This is because all the
attributes are used at least twice (condition C5). Then, the attributes
in row $l$ are also required by some other item(s) and rows A and B
require the same combination of items. Thus, the corresponding entries
of all the column vectors of $T_{c}(Q)$ are different by a factor of $c_{l}$.

For $T(Q')$, rows A and B are also identical. This is because row $h$
and row $l$ have identical attribute requirements. Then, thus, the
corresponding entries of all the column vectors of $T_{c'}(Q)$ are
different by a factor of $c_{l}'$. Thus, $c'_{l}$ and $c_{l}$ must be
identical otherwise $T_{c}(Q) \pp^{*}$ is not in the column space of
$T_{c'}(Q)$.

Similarly, we obtain that $c_{h} =c'_{h}$. Given that $c_{h}= c_{h}' $
and $c_{l} = c_{l}'$, we now consider row $h$ and row $l$. Notice that
all the column vectors in $T_{c'}(Q')$ have their entries in row $h$
and row $l$ different by a factor of $c_{h}/c_{l}$. On the other hand,
the $h$ and $l$th entries of $T_{c}(Q)\pp^{*}$ are NOT different by a
factor of $c_{h}/c_{l}$ as long as the proportion of $p_{1}$ is
positive. Thereby, we conclude this case.


\item[(b)] Consider the following two types of sub-populations: people
who are able to answer item(s) in row $l$ ($p_{1}$) only and people who
are able to answer items in both row $h$ and row $l$ ($p_{2}$). Similar
to the analysis of (a), the sub-matrices look like:
\begin{center}
\begin{tabular}{l@{\quad}l@{\quad}l}
\multicolumn{3}{c}{\hspace*{26.5pt}$T_{c}(Q)$} \\
& $p_{1}$ & $p_{2}$  \\
row $h$ & $0$ & $c_{h}$  \\
row $l$ & $c_{l}$ & $c_{l}$ 
\end{tabular}
\quad
\begin{tabular}{l@{\quad}l@{\quad}l}
\multicolumn{3}{c}{\hspace*{31.5pt}$T_{c^{\prime}}(Q^{\prime})$} \\
& $p_{1}$ & $p_{2}$  \\
row $h$ & $0$ & $c_{h}^{\prime}$  \\[2pt]
row $l$ & $0$ & $c_{l}^{\prime}$  
\end{tabular}
\end{center}
With exactly the same argument as in (a), we conclude that $c_{j} =
c'_{j}$, $c_{h}=c'_{h}$, and further $T_{c}(Q)\pp^{*}$ is not in the
column space of $T_{c'}(Q')$.

\item[(c)] Consider the following three types of sub-populations:
people who are able to answer item(s) in row $l$ only ($p_{1}$), people
who are able to answer item(s) in row $h$ only ($p_{2}$), and people
who are able to answer items in both row $h$ and row $l$ ($p_{3}$).
The sub-matrices look like:
\begin{center}
\begin{tabular}{l@{\quad}l@{\quad}l@{\quad}l}
\multicolumn{4}{c}{\hspace*{14pt}$T_{c}(Q)$} \\
& $p_{1}$ & $p_{2}$ & $p_{3}$  \\
row $h$ & $0$ & $c_{h}$ & $c_{h}$ \\
row $l$ & $c_{l}$ & $0$ & $c_{l}$ \\
row $l\wedge h$ & $0$ & $0$ & $c_{h}c_{l}$   
\end{tabular}
\quad
\begin{tabular}{l@{\quad}l@{\quad}l@{\quad}l}
\multicolumn{4}{c}{\hspace*{13pt}$T_{c^{\prime}}(Q^{\prime})$}\\
& $p_{1}$ & $p_{2}$ & $p_{3}$  \\
row $h$ & $0$ & $c_{h}^{\prime}$ & $c_{h}^{\prime}$ \\[2pt]
row $l$ & $0$ & $c_{l}^{\prime}$ & $c_{l}^{\prime}$ \\[2pt]
row $l\wedge h$ & $0$ & $c_{h}^{\prime}c_{l}^{\prime}$ &
$c_{h}^{\prime
}c_{l}^{\prime}$  
\end{tabular}
\end{center}
With the same argument, we have that $c_{l} = c'_{l}$ and
$c_{h}=c'_{h}$. On considering row $h$ and row $l\wedge h$, we conclude
that $T_{c}(Q)\pp^{*}$ is not in the column space of $T_{c'}(Q')$.
\qed
\end{enumerate}
%
%
%
%
\end{longlist}
\noqed
\end{pf*}

\begin{pf*}{Proof of Proposition~\ref{PropIncomp}}
$T(\cdot)$ is arranged as in \eqref{A}.
%
%
Consider $Q'$ such
that $Q'_{1:k}$ is incomplete. We discuss the following situations.
\begin{enumerate}
\item There are two row vectors, say the $h$th and $l$th row vectors
($1\leq i,j \leq
k$), in $Q_{1:k}^{\prime}$ that are identical. Equivalently, two
items require exactly the same attributes according to $Q'$. With
exactly the same argument as in the previous proof, under condition C5,
we have that $c_{h} = c'_{h}$ and $c_{l} = c_{l}'$. We now consider the
rows corresponding to $l$ and $l\wedge h$. Note that the elements
corresponding to row $l$ and row $l\wedge h$ for all the vectors in the
column space of $T_{c'}(Q')$ are different by a factor of $c_{h}$.
However, the corresponding elements in the vector $T_{c}(Q)\pp^{*}$
are NOT different by a factor of $c_{h}$ as long as the population is
fully diversified.

%

\item No two row vectors in $Q_{1:k}^{\prime}$ are
identical. Then, among the first $k$ rows of $Q^{\prime}$ there is
at least one row vector containing two or more nonzero entries.
That is, there exists $1\leq i\leq k$ such that
\[
\sum_{j=1}^{k}Q_{ij}^{\prime}>1.
\]
This is because if each of the first $k$ items requires only one
attribute and $Q_{1:k}^{\prime}$ is not complete, there are at
least two items that require the same attribute. Then, there are two
identical row vectors in $ Q_{1:k}^{\prime}$ and it belongs to the
first situation. We define
\[
a_{i}=\sum_{j=1}^{k}Q_{ij}^{\prime},
\]
the number of attributes required by item $i$ according to $Q'$.

Without loss of generality, assume $a_{i}>1$ for $i=1,\ldots,n$ and
$a_{i}=1$ for $i=n+1,\ldots,k$. Equivalently, among the first $k$
items, only the first $n$ items require more than one attribute
while the $(n+1)$th through the $k$th items require only one
attribute each, all of which are distinct. Without loss of
generality, we assume $Q_{ii}^{\prime}=1$ for $i=n+1,\ldots,k$ and
$Q_{ij}=0$ for $i=n+1,\ldots,k$ and $i\neq j$.
\begin{enumerate}[(b)]
\item[(a)]\label{a} $n=1$. Since $a_{1}>1$, there exists an $l>1$ such
that $Q'_{1l}=1$. We now consider rows $1$ and $l$. With the same
argument as before (i.e., the attribute required by row $l$ is also
required by item 1 in $Q'$), we have that $c_{l} = c_{l}'$ (be careful
that we cannot claim that $c_{1} = c_{1}'$). We now consider the rows 1
and $1\wedge l$. Note that in $T_{c'}(Q')$ these two rows are different
by a factor of $c_{l}$; while the corresponding entries in $T_{c}(Q)\pp
^{*}$ are NOT different by a factor of $c_{l}$. Thereby, we conclude
the result in this situation.


\item[(b)]$n>1$ and there exists $j>n$ and $i\leq n$ such that $Q'_{ij}=1$.
The argument is identical to that in (2a).


\item[(c)]$n>1$ and for each $j>n$ and $i\leq n$, $Q'_{ij}=0$. Let the
$i^{\ast}$th
row in $T(Q')$ correspond to $I_{1}\wedge\cdots\wedge I_{n}$. Let
the $i_{h}^{\ast}$th row in $T(Q')$ correspond to $I_{1}\wedge
\cdots\wedge I_{h-1}\wedge I_{h+1}\wedge\cdots\wedge I_{n}$ for
$h=1,\ldots,n$.

We claim that there exists an $h$ such that the $i^{\ast}$th row and
the $%
i_{h}^{\ast}$th row are identical in $T(Q^{\prime})$, that is%
%
\setcounter{equation}{0}
\begin{equation}\label{n1}
B_{Q^{\prime}}(I_{1}\wedge\cdots\wedge I_{h-1}\wedge I_{h+1}\wedge
\cdots\wedge I_{n})=B_{Q^{\prime}}(I_{1}\wedge\cdots\wedge I_{n}).
\end{equation}
If the above claim is true, then the attributes required by item $h$
have been required by some other items. Then, we conclude that $c_{h}$
and $c'_{h}$ must be identical. In addition, rows in $T_{c'}(Q')$
corresponding to $I_{1}\wedge\cdots\wedge I_{h-1}\wedge
I_{h+1}\wedge\cdots\wedge I_{n}$ and $I_{1}\wedge\cdots\wedge
I_{n}$ are different by a factor of $c_{h}$. On the other hand, the
corresponding entries in $T_{c}(Q)\pp^{*}$ are NOT different by a
factor of $c_{h}$. Then, we are able to conclude the results for all
the cases.

In what follows, we prove the claim in \eqref{n1} by contradiction.
Suppose that there does not
exist such an $h$. This is equivalent to saying that for each $j\leq
n$ there exists an $\alpha_{j}$ such that $Q'_{j\alpha_{j}}=1$ and
$Q'_{i\alpha_{j}}=0$ for all $1\leq i\leq n$ and $i\neq j$.
Equivalently, for each $j\leq n$, item $j$ requires at least one
attribute that is not required by other first $n$ items. Consider
\[
\mathcal{C}_{i}=\{j\dvtx \mbox{there exists $i\leq i' \leq n$ such that
$Q'_{i' j}=1$}\}.
\]
Let $\#(\cdot)$ denote the cardinality of a set. Since for each
$i\leq n$ and $j>n$, $Q'_{ij}=0$, we have that
$\#(\mathcal{C}_{1})\leq n$. Note that $Q'_{1\alpha_{1}}=1$ and
$Q'_{i\alpha_{1}}=0$ for all\vspace*{1.5pt} $2\leq i\leq n$. Therefore, $\alpha
_{1}\in\mathcal{C}_{1}$ and $\alpha_1 \notin\mathcal{C}_{2}$.
Therefore, $\#(\mathcal{C}_{2})\leq n-1$. By a similar argument and
induction, we have that $a_{n}=\#(\mathcal{C}_{n})\leq1 $. This
contradicts the fact that $a_{n}>1$. Therefore, there exists an $h$
such that \eqref{n1} is true. As for $T(Q)$, we have that
\[
B_{Q}(I_{1}\wedge\cdots\wedge I_{h-1}\wedge I_{h+1}\wedge
\cdots\wedge I_{n})\neq B_{Q}(I_{1}\wedge\cdots\wedge I_{n}).
\]
\end{enumerate}

\end{enumerate}

Summarizing the cases in 1, (2a), (2b) and (2c), we conclude the proof.
\end{pf*}
\end{appendix}

\section*{Acknowledgements}
 This research was supported in part by
Grants NSF CMMI-1069064, SES-1123698, Institute of Education Sciences
R305D100017 and NIH R37 GM047845.



\printhistory

\end{document}